\documentclass[final,1p,times]{elsarticle}
\usepackage{amssymb}
\usepackage{amssymb,amsmath}
\usepackage{graphicx,epsfig}
\usepackage{color}
\usepackage {multicol}
\usepackage{subfig}
\usepackage{varwidth}
\usepackage{float}
\usepackage{graphicx}
\usepackage{booktabs}   
\usepackage{array}      
\usepackage{booktabs}   
\usepackage{array}      
\usepackage{multirow}
\usepackage{amsmath}
\usepackage{breqn}
\usepackage{natbib}
\usepackage{amssymb}
\usepackage{color}
\usepackage{subfig}
\usepackage{amsmath,mathtools}
\newtheorem{theorem}{Theorem}
\newtheorem{lemma}{Lemma}

\newtheorem{corollary}{Corollary}
\newtheorem{proof}{Proof}

\newtheorem{remark}{Remark}
\newcommand{\bea}{\begin{eqnarray}}
\newcommand{\eea}{\end{eqnarray}}

\allowdisplaybreaks
\begin{document}

\begin{frontmatter}
\title{A study on reliability of a $k$-out-of-$n$ system equipped with a cold standby component based on copula}
\author{Achintya Roy, Nitin Gupta }
\address{\normalfont Department of Mathematics,
Indian Institute of Technology Kharagpur, West Bengal, INDIA.\\
Email: acroy21@iitkgp.ac.in (Achintya Roy), nitin.gupta@maths.iitkgp.ac.in (Nitin Gupta).}
\begin{abstract}
A $k$-out-of-$n$ system consisting of $n$ exchangeable components equipped with a single cold standby component is considered.
All the components including standby are assumed to be dependent and this dependence is modeled by a copula function.
 An exact expression for the reliability function of the considered system has been obtained. We also compute three different mean residual life functions of the system.
  Finally, some numerical examples are provided to illustrate the theoretical results. Our studies subsume some of the initial studies in the literature.
\end{abstract}
\end{frontmatter}
\section{Introduction}
\label{sec1}
Cold standby sparing is an effective technique to improve the reliability of a system. It has attracted substantial interest in the field of reliability theory due to its various applications in nuclear industry, satellites, and telecommunication systems. The purchasers or users of capital assets such as aircraft or military equipment always think about a guarantee of high availability since, the consequences of failure time can have very serious repercussions. For example, in the case of army equipment failure time may lead to mission failure.
A cold standby component is inactive and has no failure rate while in the standby condition. An inactive cold standby component becomes active at the time the principal system breaks down and keep the system functioning with the help of remaining active components of the failed principal system.
It is usually not possible to repair the engines of spacecrafts, satellites etc. on and during a space mission.
Hence, it may be desirable to reduce the risk of engines's failure by using cold standby sparing.
Many statisticians have investigated coherent systems with cold standby components when components are independent and identical (see \citet{eryilmaz2012mean}, \citet{eryilmaz2014study},\citet{wang2016conditional}, \citet{mirjalili2017note}, \citet{roy2019reliability}, and \citet{Roy2020reliability}). In real life situations the components which form a system may be dependent. The dependence may occur due to sharing same workloads such as heat and tasks. The standby components which are connected to a system become dependent with system's components due to their placement in the system and system's design. Some studies in this case are found in  \citet{navarro2007properties}, \citet{bairamov2011order}, \citet{sadegh2011note}, \citet{eryilmaz2012reliability1},  \citet{khanjari2016distribution}, \citet{yongjin2018copula}. These studies cover the investigation of lifetime characteristics such as reliability function, mean residual life function, mean time to failure and stochastic ordering results associated with system's lifetime. \par
A set of random variables can be either independent of each other or can be dependent.
In case of independent random variables, the joint distribution function can be written as
a product of the individual distribution functions. Whereas if the random variables are
dependent we require a copula function to determine the joint distribution function.
The concept of copulas was introduced by Sklar (1959).
 His fundamental theorem is given below: \\
\bf{Theorem A} [\normalfont Sklar's Theorem 1959]
Let $F$ be a $n$-dimensional cumulative distribution function with marginal
distribution functions $F_{i}, i = 1,\ldots,n$. Then there exists a copula $C$ such
that
$$F(y_{1},\ldots,y_{n})=C(F_{1}(y_{1})),\ldots,F_{n}(y_{n})).$$
Conversely, for any univariate cumulative distribution functions $F_{i}, i = 1,\ldots,n$
and any copula $C$, the function $F$ is a $n$-dimensional cumulative distribution
function with marginal distribution functions $F_{i}, i = 1,\ldots,n$.
Furthermore, if $F_{1},\ldots, F_{n}$ are continuous,
then C is unique. An extensive review of the copula function and its applications is given in \citet{nelsen2007introduction}.
\par
 As a generalization of series and parallel system, a $k$-out-of-$n$ system containing $n$ components is said to be operational  iff at least  $k~(k \leq n)$ components are operational (see  \citet{barlow1975statistical}). There is quite rich literature on reliability properties of $k$-out-of-$n$ systems consisting of $n$ independent components . Some recent results in this area are in \citet{bairamov2002residual}, \citet{asadi2005note},  \citet{li2006some}, \citet{navarro2010comparisons}, \citet{eryilmaz2011dynamic}, \citet{zhao2012optimal}, \citet{levitin2014cold}, and \citet{zhao2017optimal}. \citet{asadi2005note} obtained the mean residual life function $E(X_{n-k+1:n}-t|X_{1:n}>t); ~t\geq0.$  And the mean residual life function $E(X_{n-k+1:n}-t|X_{l:n}>t)$;~$1 \leq l<k  \leq n$  has been  investigated by \citet{li2006some}. \citet{li2006some} actually generalized the work of \citet{asadi2005note}.
 In reliability theory, the assumption of dependence of lifetimes of components is more realistic than assumption of independence.
  Exchangeability is a type of dependence which has attracted the interest of many statisticians in recent times.
 A set of random variables $Z_{1},\ldots,Z_{n}$ are said to be exchangeable if
 $$P(Z_{1}\leq z_{1},\ldots,Z_{n} \leq z_{n})=P(Z_{\alpha(1)}\leq z_{1},\ldots,Z_{\alpha(n)} \leq z_{n}),$$
 where $\alpha=(\alpha(1),\ldots,\alpha(n))$ is an arbitrary permutation of $\{1,\ldots,n\},$ that is, the joint distribution of
  $Z_{1},\ldots,Z_{n}$ is symmetric in  $z_{1},\ldots,z_{n}.$ It is clear~that $Z_{1},\ldots,Z_{n}$ are identically distributed.
  Note that independent and identically distributed random variables are also exchangeable random variables.
  The reliability properties of $k$-out-of-$n$ system with exchangeable components have been investigated by
   \citet{navarro2005note}, \citet{navarro2007properties},\citet{sadegh2011note}, \citet{khanjari2016distribution}, \citet{eryilmaz2020}. \par
 \citet{eryilmaz2012mean} investigated the importance of a single cold standby component towards increasing the $k$-out-of-$n$ system's reliability when components are independent and identically distributed. Here the standby component is applied when the $k$-out-of-$n$ system fails. He obtained the reliability function and three different mean residual functions of the $k$-out-of-$n$ system which is equipped with a single cold standby component. \citet{eryilmaz2012reliability1} studied reliability properties of a two-component series system equipped with single cold standby component when series system's components are exchangeable. The authors also assumed that all the components including standby component are identical and dependent. \citet{yongjin2018copula} have enlarged their study by considering a multi components parallel system.
 Motivated by their work, in this article, we consider a $k$-out-of-$n$ system consisting of $n$ exchangeable components and equipped with a single cold standby component. We assume that all the components including standby are dependent, and we use a given copula function to design this dependency. Let $Z$ and  $Z_{i}$ denote the lifetimes of the cold standby and  the $i$th $(i=1,2,\ldots,n)$ component of the k-out-of-n system respectively. Throughout this article, we denote distributions functions $P(Z_{i}\leq z)=F(z),~i=1,2,\ldots,n$, and $P(Z\leq z)=G(z),~z \geq 0.$ Let $f$ and $g$ are probability density functions of $Z_{i}$ and $Z$ respectively.
Let $C$ be a copula function of the lifetimes $Z_{1},\ldots,Z_{n}$ and $Z,$ then  by using Sklar's theorem (for more details about Sklar's theorem, one can see Theorem 2.10.9 in \citet{nelsen2007introduction}),
\begin{align}
P(Z_{1}\leq z_{1},\ldots,&Z_{n}\leq z_{n}, Z\leq z)=C(F(z_{1}),\ldots,F(z_{n}),G(z));
\end{align}
 \par
If $h(z_{1},\ldots,z_{n},z)$ represents the joint density function of lifetimes $Z_{1},\ldots,Z_{n}$ and $Z,$ then the density function, according to Sklar's theorem, can be written using copula function $C$  as
\begingroup
\allowdisplaybreaks
\begin{align}\label{3}
 h(z_{1},\ldots,z_{n},z)=c(F(z_{1}),\ldots,F(z_{n}),G(z))\Bigg[\prod_{i=1}^{n} f(z_{i})\Bigg] g(z),
\end{align}
\endgroup
where    $c:\textbf{I}^{n}\rightarrow \textbf{I},~\textbf{I}=[0,1],$  the density function of copula $C$ as follows
 \begin{align} c(F(z_{1}),\ldots,F(z_{n}),G(z))=\frac{\partial^{n+1} C(F(z_{1}),\ldots,F(z_{n}),G(z))}{\partial F(z_{1})\ldots\partial F(z_{n})\partial G(z)}.
 \end{align}
 It is clear that $\frac{\partial^{n+1} C(F(z_{1}),\ldots,F(z_{n}),G(z))}{\partial F(z_{1})\ldots\partial F(z_{n})\partial G(z)}=1$
 when $Z_{1},\ldots,Z_{n},$ and $Z$ are all independent.\\
  As components' lifetimes of the  $k$-out-of-$n$ system  are denoted by  $Z_{1},\ldots,Z_{n}$, then $Z_{n-k+1:n}$ is the failure time of the k-out-of-n system. And the cold standby component starts working at the random time $Z_{n-k+1:n}$.
Therefore, the considered system's lifetime is denoted by
\begin{align}
&T=Z_{n-k+1:n}+ \min\{Z_{n-k+2:n}-Z_{n-k+1:n},Z\},~k \geq 2;\nonumber\\ &T=Z_{n:n}+Z.
\end{align}
A system, which consists of components, is called coherent system if it is monotone and none of its components is irrelevant.
Standby component $Z$ is irrelevant to the performance of the considered system before random time $Z_{n-k+1:n}.$ In this sense,
the considered system with lifetimes $Z_{1},\ldots,Z_{n},~Z$ cannot be considered as a coherent system.
The idea of mean residual life function is an important reliability characteristics that has been extensively used in reliability analysis.
It is a useful tool for studying ageing properties of a system.
In this article, we study three different mean residual life (MRL) functions which are defined by
\begin{align} \label{mrl1}
\Psi_{1}(t)=E(T-t|T>t),
\end{align}
\begin{align} \label{mrl2}
\Psi_{2}(t)=E(T-t|Z_{n-k+1:n}>t)
\end{align}
\begin{align} \label{mrl3}
\Psi_{3}(t)=E(T-t|Z_{1:n}>t),
\end{align}
for $t\geq 0.$ \\
The rest of the article is organized as follows:
In Section 2,  first  we  derive  explicit expressions for $P(T>t)$ by using concept of order statistics in terms of copula functions.
After that, we have obtained representation of the mean residual functions defined by (\ref{mrl1})-(\ref{mrl3}). Also, some numerical examples are presented.
\section{Main results}
In this section we compute the reliability function of $T.$
\begin{theorem} \label{thm1}
 \normalfont {Reliability function  of  $T$ is given by}
 \begingroup
\allowdisplaybreaks
\begin{align}\label{m1}
P(T>s)=&1-\sum _{i=n-k+1}^{n} (-1)^{i-n+k-1}\binom{n}{i}\binom{i-1}{n-k}C(\underbrace{F(s),\ldots,F(s)}_{\text{i}},1,\ldots,1)+
\frac{1}{B(n-k+1,k)} \nonumber\\&\times\int_{0}^{s}\underbrace{\int_{0}^{z_{n}}\ldots\int_{0}^{z_{n}}}_{\text{n-k}}\underbrace{ \int_{s}^{\infty}\ldots\int_{s}^{\infty}}_{\text{k-1}}\int_{s-z_{n}}^{\infty}  \frac{\partial^{n+1} C(F(z_{1}),\ldots,F(z_{n}),G(z))}{\partial F(z_{1})\ldots\partial F(z_{n})\partial G(z)}\Bigg(\prod_{i=1}^{n} f(z_{i})\Bigg)~\nonumber\\&\times g(z) ~ dzdz_{1} \ldots dz_{k-1}dz_{k}\ldots dz_{n-1}dz_{n}.
\end{align}
\endgroup
\end{theorem}
\begin{proof}
\normalfont
It is clear that,
\begin{align}\label{k8}
P(T>s, Z_{n-k+1:n}>t)&=P(Z_{n-k+1:n}>s)\nonumber\\ &+P(Z_{n-k+1:n}+\min\{Z_{n-k+2:n}-Z_{n-k+1:n},Z\}>s, t\leq Z_{n-k+1:n} \leq s)
\end{align}
Reliability function of the $k$-out-of-$n$ system is given by
\begin{align}\label{k8111}
P(Z_{n-k+1:n}>s)&=1-\sum _{i=n-k+1}^{n} (-1)^{i-n+k-1}\binom{n}{i}\binom{i-1}{n-k}P(Z_{1}\leq s,\ldots,Z_{i} \leq s)\nonumber\\&
=1-\sum _{i=n-k+1}^{n} (-1)^{i-n+k-1}\binom{n}{i}\binom{i-1}{n-k}C(\underbrace{F(s),\ldots,F(s)}_{\text{i}},1,\ldots,1)
\end{align}
Since $Z_{1},\ldots,Z_{n}$ are identical and $Z_{n-k+1:n} \in \{Z_{1},\ldots, Z_{n}\}$, using the total probability law, we have
\begin{align}\label{k81}
 P(Z_{n-k+1:n}+\min\{Z_{n-k+2:n}&-Z_{n-k+1:n},Z\}>s, t\leq Z_{n-k+1:n} \leq s)\nonumber\\&
 =nP(Z_{n}+Z>s, t\leq Z_{n}\leq s,Z_{n-k+2:n}>s,Z_{n-k+1:n}=Z_{n})
\end{align}
$Z_{n-k+1:n}=Z_{n}$ implies that $k-1$ of $Z_{1},\ldots, Z_{n-1}$ are greater than $Z_{n}$ and others are less than  $Z_{n}$. Since $Z_{1},\ldots,Z_{n}$ are identical, we can assume that $Z_{1},\ldots, Z_{k-1}$ are greater than $Z_{n}$ and others are less than $Z_{n}$.
Therefore,
\begin{align}\label{k811}
 &P(Z_{n-k+1:n}+\min\{Z_{n-k+2:n}-Z_{n-k+1:n},Z\}>s,t\leq Z_{n-k+1:n} \leq s)\nonumber\\&
 =n \binom{n-1}{n-k}P(Z_{1}>s,\ldots, Z_{k-1}> s,Z_{n}+Z>s,Z_{k}< Z_{n},\ldots,Z_{n-1}< Z_{n}, t\leq Z_{n}<s)
  \nonumber\\&=\frac{1}{B(n-k+1,k)}  P(Z_{1}>s,\ldots, Z_{k-1}> s,Z_{n}+Z>s,Z_{k}< Z_{n},\ldots,Z_{n-1}< Z_{n}, t\leq Z_{n}<s)\nonumber\\&
  =\frac{1}{B(n-k+1,k)}\int_{t}^{s}\underbrace{\int_{0}^{z_{n}}\ldots\int_{0}^{z_{n}}}_{\text{n-k}}\underbrace{ \int_{s}^{\infty}\ldots\int_{s}^{\infty}}_{\text{k-1}}\int_{s-z_{n}}^{\infty}  \frac{\partial^{n+1} C(F(z_{1}),\ldots,F(z_{n}),G(z))}{\partial F(z_{1})\ldots\partial F(z_{n})\partial G(z)} \times \nonumber\\&~~~~~~~~~~~~~~~~~~~~~~~~~~~~~~~~~~~~~~~~~~~~~~~~~~~~~~~~~~~~~\Bigg(\prod_{i=1}^{n} f(z_{i})\Bigg)~ g(z) ~ dzdz_{1} \ldots dz_{k-1}dz_{k}\ldots dz_{n-1}dz_{n}.
   \end{align}
Therefore, for $t=0,$ from equations (\ref{k8})-(\ref{k811}), we have equation (\ref{m1}).
\end{proof}
\begin{remark}
\normalfont
If all the components are independent, that is,
if $\frac{\partial^{n+1} C(F(z_{1}),\ldots,F(z_{n}),G(z))}{\partial F(z_{1})\ldots\partial F(z_{n})\partial G(z)}=1,$ then the reliability function of $T$ is found to be
 \begingroup
\allowdisplaybreaks
\begin{align}
P(T>s)=&1-\sum _{i=n-k+1}^{n} (-1)^{i-n+k-1}\binom{n}{i}\binom{i-1}{n-k}F^{i}(s)+
\frac{1}{B(n-k+1,k)}\times \nonumber\\&\int_{0}^{s}\underbrace{\int_{0}^{z_{n}}\ldots\int_{0}^{z_{n}}}_{\text{n-k}}\underbrace{ \int_{s}^{\infty}\ldots\int_{s}^{\infty}}_{\text{k-1}}\int_{s-z_{n}}^{\infty} \Bigg(\prod_{i=1}^{n} f(z_{i})\Bigg) g(z)dzdz_{1} \ldots dz_{k-1}dz_{k}\ldots dz_{n-1}dz_{n}\nonumber\\&=P(Z_{n-k+1:n}>s)+\frac{\overline{F}^{k-1}(s)}{B(n-k+1,k)} \int_{0}^{s} \overline{G}(s-z)F^{n-k}(z)f(z)dz,
\end{align}
\endgroup
for $s>0.$ This case is considered in \citet{eryilmaz2012mean}.
\end{remark}
\begin{remark}
\normalfont
As $P(Z_{n-k+1:n}>s)$ is the reliability function of the $k$-out-of-$n$ system, then from equations (\ref{m1}) and (\ref{k8111}),
we can write  the contribution of the cold standby component to improve the reliability of the considered system is given by
 \begingroup
\allowdisplaybreaks
\begin{align*}
&\frac{1}{B(n-k+1,k)}\int_{0}^{s}\underbrace{\int_{0}^{z_{n}}\ldots\int_{0}^{z_{n}}}_{\text{n-k}}\underbrace{ \int_{s}^{\infty}\ldots\int_{s}^{\infty}}_{\text{k-1}}\int_{s-z_{n}}^{\infty}  \frac{\partial^{n+1} C(F(z_{1}),\ldots,F(z_{n}),G(z))}{\partial F(z_{1})\ldots\partial F(z_{n})\partial G(z)}\Bigg(\prod_{i=1}^{n} f(z_{i})\Bigg)~\nonumber\\&\times g(z) ~ dzdz_{1} \ldots dz_{k-1}dz_{k}\ldots dz_{n-1}dz_{n}.
\end{align*}
\endgroup
\end{remark}
\begin{corollary} \label{corollary1}
\normalfont
By using equation (\ref{m1}), the usual mean residual function of $T$ can be obtained from
\begingroup
\allowdisplaybreaks
\begin{align} \label{k36}
&\Psi_{1}(t)=E(T-t|T>t)
=\frac{1}{P(T>t)}\int_{t}^{\infty}P(T>s)ds,
\end{align}
\endgroup
and $\Psi_{1}(0)=E(T)$ gives the mean time to failure of $T$. \\
\end{corollary}
\begin{corollary} \label{corollary2}
\normalfont
Next, we compute the mean residual function of $T$ under ~the condition $Z_{n-k+1:n}>t,$ i.e., the $k$-out-of-$n$ system active at time t.
\begin{align}\label{mrl2}
\Psi_{2}(t)&=E(T-t|Z_{n-k+1:n}>t)\nonumber\\&
=\frac{1}{P(Z_{n-k+1:n}>t)}\int_{0}^{\infty}P(T>t+x,~Z_{n-k+1:n}>t)dx\nonumber\\&
=\frac{1}{P(Z_{n-k+1:n}>t)}\Bigg[\int_{t}^{\infty}P(Z_{n-k+1:n}>x) dx+
\nonumber\\&\int_{0}^{\infty}P(Z_{n-k+1:n}+\min\{Z_{n-k+2:n}-Z_{n-k+1:n},Z\}>t+x,t<Z_{n-k+1:n} \leq t+x) dx\Bigg]\nonumber\\&
=E(Z_{n-k+1:n}-t|Z_{n-k+1:n}>t)+\frac{1}{P(Z_{n-k+1:n}>t)}\times \nonumber\\& \Bigg[\frac{1}{B(n-k+1,k)}
\int_{0}^{\infty}\int_{t}^{t+x}\underbrace{\int_{0}^{z_{n}}\ldots\int_{0}^{z_{n}}}_{\text{n-k}}\underbrace{ \int_{t+x}^{\infty}\ldots\int_{t+x}^{\infty}}_{\text{k-1}}\int_{t+x-z_{n}}^{\infty}  \frac{\partial^{n+1} C(F(z_{1}),\ldots,F(z_{n}),G(z))}{\partial F(z_{1})\ldots\partial F(z_{n})\partial G(z)} \times \nonumber\\&
~~~~~~~~~~~~~~~~~~~~~~~~~~~~~~~~~~~~~~~~~~~~~\Bigg(\prod_{i=1}^{n} f(z_{i})\Bigg)~ g(z) ~ dzdz_{1} \ldots dz_{k-1}dz_{k}\ldots dz_{n-1}dz_{n} dx\Bigg].
\end{align}
Now if $Z_{1},\ldots,Z_{n},Z$ are independent, then
\begin{align}
E(T-t|&Z_{n-k+1:n}>t)
=E(Z_{n-k+1:n}-t|Z_{n-k+1:n}>t)+\frac{1}{P(Z_{n-k+1:n}>t)}\times \nonumber\\& \Bigg[\frac{1}{B(n-k+1,k)}
\int_{0}^{\infty}\int_{t}^{t+x} \overline{F}^{k-1}(t+x)\overline{G}(t+x-z_{n}) F^{n-k}(z_{n})f(z_{n}) dz_{n} dx\Bigg].
\end{align}
From equation (\ref{mrl2}), the contribution of the cold standby component to the mean time to failure of the considered system is given by
\begin{align}
&E(T-Z_{n-k+1:n})\nonumber\\&
=\frac{1}{B(n-k+1,k)}
\int_{0}^{\infty}\int_{0}^{x}\underbrace{\int_{0}^{z_{n}}\ldots\int_{0}^{z_{n}}}_{\text{n-k}}\underbrace{ \int_{x}^{\infty}\ldots\int_{x}^{\infty}}_{\text{k-1}}\int_{x-z_{n}}^{\infty}  \frac{\partial^{n+1} C(F(z_{1}),\ldots,F(z_{n}),G(z))}{\partial F(z_{1})\ldots\partial F(z_{n})\partial G(z)} \times \nonumber\\&
~~~~~~~~~~~~~~~~~~~~~~~~~~~~~~~~~~~~~~~~~~~~~\Bigg(\prod_{i=1}^{n} f(z_{i})\Bigg)~ g(z) ~ dzdz_{1} \ldots dz_{k-1}dz_{k}\ldots dz_{n-1}dz_{n} dx.
\end{align}
\end{corollary}
\qed \\
\citet{asadi2006mean} investigated the MRL function $E(Z_{n-k+1:n}-t|Z_{1:n}>t)$ when  $Z_{1},\ldots,Z_{n}$ are independent and identical. They proved that
 \begin{align}
 E(Z_{n-k+1:n}-t|Z_{1:n}>t)
 =\sum_{m=0}^{n-k}\sum_{i=0}^{m} {n \choose m} {m \choose i}(-1)^{i}\frac{\int_{t}^{\infty} \overline{F}^{n-m+i}(z) dz}{\overline{F}^{n-m+i}(t)};~~t>0.
 \end{align}
 Now, we obtain $E(Z_{n-k+1:n}-t|Z_{1:n}>t)$ in terms of the joint cumulative distribution of $Z_{i}'$s
 when $Z_{1},\ldots,Z_{n}$ are exchangeable and dependent.
 First, we state and prove a lemma which is used to obtain the expression of $E(Z_{n-k+1:n}-t|Z_{1:n}>t).$
\begin{lemma}
\normalfont For $i\geq 1;$
\begin{align} \label{lemma 1}
P(t\leq Z_{1}&\leq t+x,\ldots,t \leq Z_{i} \leq t+x, Z_{i+1}\geq t+x,\ldots,Z_{n}>t+x)\nonumber\\&
=P(Z_{1}\leq t+x,\ldots, Z_{i} \leq t+x)\nonumber\\&
-\sum _{j=1}^{i} (-1)^{j+1} \binom{i}{j}P(Z_{1}\leq t,\ldots, Z_{j} \leq t, Z_{j+1}\leq t+x,\ldots,Z_{i}\leq t+x)\nonumber\\&-\sum _{j=1}^{n-i} (-1)^{j+1} \binom{n-i}{j} P(Z_{1}\leq t+x,\ldots, Z_{i+j}\leq t+x)\nonumber\\&
+\sum _{j=1}^{i}\sum _{m=1}^{n-i} (-1)^{j+m}\binom{i}{j} \binom{n-i}{m}
P(Z_{1}\leq t,\ldots,Z_{j}\leq t, Z_{j+1} \leq t+x,\ldots, Z_{i+m}\leq t+x)
\end{align}
\end{lemma}
\begin{proof}
\normalfont Let \\
P=$(Z_{1}\leq t)\cup\ldots \cup  (Z_{i}\leq t) \cup (Z_{i+1}\leq t+x)\cup\ldots \cup  (Z_{n}\leq t+x),$
 and, \\
Q=$(Z_{1}\leq t+x,\ldots,Z_{i} \leq t+x).$\\
Then, $P^{c} \cap Q=(t\leq Z_{1}\leq t+x,\ldots,t \leq Z_{i} \leq t+x, Z_{i+1}\geq t+x,\ldots,Z_{n}>t+x),$
and, \\
$P \cap Q=(Z_{1}\leq t, Z_{2}\leq t+x,\ldots,Z_{i} \leq t+x)\cup\ldots\cup(Z_{1}\leq t+x,\ldots,Z_{i-1} \leq t+x,Z_{i}\leq t) \\
~~~~~~~~~~~~~~~~\cup(Z_{1}\leq t+x,\ldots,Z_{i+1} \leq t+x)\cup\ldots\cup(Z_{1}\leq t+x,\ldots,Z_{i} \leq t+x,Z_{n} \leq t+x).$\\
Note that $\normalfont P(P^{c} \cap Q)=P(Q)-P(P\cap Q).$
Hence by applying the principle of inclusion-exclusion, we get
the equation (\ref{lemma 1}).
\end{proof}
In the next lemma we compute $E(Z_{n-k+1:n}-t|Z_{1:n}>t)$.
\begin{lemma}
\begin{align} \label{lemma 2}
E(&Z_{n-k+1:n}-t|Z_{1:n}>t)\nonumber\\&=\frac{1}{P(Z_{1:n}>t)}\int_{0}^{\infty}\Bigg[1-\sum _{j=1}^{n}(-1)^{j-1}\binom{n}{j}P(Z_{1}\leq t+x,\ldots,Z_{j}\leq t+x)\Bigg]dx\nonumber\\&+\frac{1}{P(Z_{1:n}>t)}\int_{0}^{\infty}\Bigg[\sum _{i=1}^{n-k} \binom{n}{i}\Bigg[P(Z_{1}\leq t+x,\ldots, Z_{i} \leq t+x)\nonumber\\&
-\sum _{j=1}^{i} (-1)^{j+1} \binom{i}{j}P(Z_{1}\leq t,\ldots, Z_{j} \leq t, Z_{j+1}\leq t+x,\ldots,Z_{i}\leq t+x)\nonumber\\&-\sum _{j=1}^{n-i} (-1)^{j+1} \binom{n-i}{j} P(Z_{1}\leq t+x,\ldots, Z_{i+j}\leq t+x)\nonumber\\&
+\sum _{j=1}^{i}\sum _{m=1}^{n-i} (-1)^{j+m}\binom{i}{j} \binom{n-i}{m}
P(Z_{1}\leq t,\ldots,Z_{j}\leq t, Z_{j+1} \leq t+x,\ldots, Z_{i+m}\leq t+x)\Bigg]dx.
\end{align}
\end{lemma}
\begin{proof}
\begin{align} \label{st}
E(Z_{n-k+1:n}-t|Z_{1:n}>t)=\frac{1}{P(Z_{1:n}>t)}\int_{0}^{\infty}P(Z_{n-k+1:n}>t+x,Z_{1:n}>t)dx
\end{align}
\normalfont
The joint reliability function of $(Z_{n-k+1:n},Z_{1:n})$ is given by
\begin{align}\label{k811110}
P(Z_{n-k+1:n}&>t+x,Z_{1:n}>t)\nonumber\\&
=\sum _{i=0}^{n-k} \binom{n}{i}P(t\leq Z_{1}\leq t+x,\ldots,t \leq Z_{i} \leq t+x, Z_{i+1}\geq t+x,\ldots,Z_{n}>t+x)\nonumber\\&
=P(Z_{1}\geq t+x,\ldots,Z_{n}>t+x)\nonumber\\&
+\sum _{i=1}^{n-k} \binom{n}{i}P(t\leq Z_{1}\leq t+x,\ldots,t \leq Z_{i} \leq t+x, Z_{i+1}\geq t+x,\ldots,Z_{n}>t+x).
\end{align}
By using equations (\ref{lemma 1}), (\ref{st}) and (\ref{k811110}), we have equation (\ref{lemma 2}).
\end{proof}
\qed \\
\begin{remark}
\normalfont
If $Z_{1},\ldots,Z_{n}$ are independent, then
\begin{align}
E(Z_{n-k+1:n}&-t|Z_{1:n}>t)\nonumber\\&
=\frac{1}{\overline{F}^{n}(t)}\int_{0}^{\infty}\Bigg[1+\sum _{j=1}^{n}(-1)^{j}\binom{n}{j}F^{i}(t+x)
+\sum _{i=1}^{n-k} \binom{n}{i}\Bigg[F^{i}(t+x)
+\sum _{j=1}^{i} (-1)^{j} \binom{i}{j}F^{j}(t)\nonumber\\&F^{i-j}(t+x)+\sum _{j=1}^{n-i} (-1)^{j} \binom{n-i}{j} F^{i+j}(t+x)+\sum _{j=1}^{i}\sum _{m=1}^{n-i} (-1)^{j+m}\binom{i}{j} \binom{n-i}{m}F^{j}(t)F^{i+m-j}(t+x)\Bigg]dx\nonumber\\&
=\frac{1}{\overline{F}^{n}(t)} \int_{0}^{\infty}\Bigg[\overline{F}^{n}(t+x)
+\sum _{i=1}^{n-k} \binom{n}{i}\Bigg[F^{i}(t+x)\sum _{j=0}^{n-i}(-1)^{j} \binom{n-i}{j} F^{j}(t+x)+
\sum _{j=1}^{i} (-1)^{j} \binom{i}{j}F^{j}(t)\nonumber\\&F^{i-j}(t+x)\Bigg[1+\sum _{m=1}^{n-i} (-1)^{m} \binom{n-i}{m}F^{m}(t+x)
\Bigg]\nonumber\\&
=\frac{1}{\overline{F}^{n}(t)} \int_{0}^{\infty}\Bigg[\overline{F}^{n}(t+x)+\sum _{i=1}^{n-k}\binom{n}{i}
\Bigg[\sum _{j=0}^{n-i}(-1)^{j} \binom{n-i}{j} F^{j}(t+x)\Bigg] \times \nonumber\\&
~~~~~~~~~~~~~~~~~~~~~~~~~~~~~~~~~~~~~~~~~~~~~~~~~~~~~~~~~~~~~~~~~~~~~~~~\Bigg[\sum _{j=0}^{i} (-1)^{j} \binom{i}{j}F^{j}(t)F^{i-j}(t+x)\Bigg]\Bigg]dx\nonumber\\&
=\frac{1}{\overline{F}^{n}(t)} \sum _{i=0}^{n-k}\binom{n}{i}\int_{0}^{\infty}(F(t+x)-F(t))^{i}\overline{F}^{n-i}(t+x)dx\nonumber\\&
=\sum _{i=0}^{n-k}\binom{n}{i}\int_{t}^{\infty}(1-\frac{\overline{F}(x)}{\overline{F}(t)})^{i}(\frac{\overline{F}(x)}{\overline{F}(t)})^{n-i}dx\nonumber\\&
=\sum_{m=0}^{n-k}\sum_{i=0}^{m} {n \choose m} {m \choose i}(-1)^{i}\frac{\int_{t}^{\infty} \overline{F}^{n-m+i}(z) dz}{\overline{F}^{n-m+i}(t)}.
\end{align}
\end{remark}
\qed \\
In terms of copula functions, the mean residual life function $E(Z_{n-k+1:n}-t|Z_{1:n}>t)$ is given by
\begin{align}
E(&Z_{n-k+1:n}-t|Z_{1:n}>t)\nonumber\\&=\frac{1}{P(Z_{1:n}>t)}\int_{0}^{\infty}\Bigg[1-\sum _{j=1}^{n}(-1)^{j+1}\binom{n}{j}
C(\underbrace{F(t+x),\ldots,F(t+x)}_{\text{j}},1,\ldots,1)\Bigg]dx\nonumber\\&
+\frac{1}{P(Z_{1:n}>t)}\int_{0}^{\infty}\Bigg[\sum _{i=1}^{n-k} \binom{n}{i}\Bigg[C(\underbrace{F(t+x),\ldots,F(t+x)}_{\text{i}},1,\ldots,1)\nonumber\\&
-\sum _{j=1}^{i} (-1)^{j+1} \binom{i}{j}C(\underbrace{F(t),\ldots,F(t)}_{\text{j}},\underbrace{F(t+x),\ldots,F(t+x)}_{\text{i-j}},1,\ldots,1)
\nonumber\\&-\sum _{j=1}^{n-i} (-1)^{j+1} \binom{n-i}{j} C(\underbrace{F(t+x),\ldots,F(t+x)}_{\text{i+j}},1,\ldots,1)\nonumber\\&
+\sum _{j=1}^{i}\sum _{m=1}^{n-i} (-1)^{j+m}\binom{i}{j} \binom{n-i}{m}
C(\underbrace{F(t),\ldots,F(t)}_{\text{j}},\underbrace{F(t+x),\ldots,F(t+x)}_{\text{i+m-j}},1,\ldots,1)\Bigg]dx.
\end{align}

In the next theorem, we obtain the mean residual function of $T$ under~~ the condition that $Z_{1:n}>t$,
\begin{theorem}
\begin{align} \label{k59}
&\Psi_{3}(t)=E(T-t|Z_{1:n}>t)=E(Z_{n-k+1:n}-t|Z_{1:n}>t)+
\frac{1}{P(Z_{1:n}>t)B(n-k+1,k)} \times \nonumber\\& \Bigg[\int_{0}^{\infty}\int_{t}^{\infty}\underbrace{\int_{t}^{z_{n}}\ldots\int_{t}^{z_{n}}}_{\text{n-k}}\underbrace{ \int_{z_{n}+x}^{\infty}\ldots\int_{z_{n}+x}^{\infty}}_{\text{k-1}}\int_{x}^{\infty}  \frac{\partial^{n+1} C(F(z_{1}),\ldots,F(z_{n}),G(z))}{\partial F(z_{1})\ldots\partial F(z_{n})\partial G(z)}\times  \nonumber\\&~~~~~~~~~~~~~~~~~~~~~~~~~~~~~~~~~~~~~~~~~~~~~~~~~~~~~~~ \Bigg(\prod_{i=1}^{n} f(z_{i})\Bigg)~ g(z) ~ dzdz_{1} \ldots dz_{k-1}dz_{k}\ldots dz_{n-1}dz_{n}dx\Bigg].
\end{align}
\end{theorem}
\begin{proof}
\normalfont
It is clear that
\begin{align} \label{k559}
E(T-t|Z_{1:n}>t)=E(Z_{n-k+1:n}-t|Z_{1:n}>t)+E(\min(Z_{n-k+2:n}-Z_{n-k+1:n},Z)|Z_{1:n}>t)).
\end{align}
We can write,
\begin{align} \label{k5590}
E(&\min(Z_{n-k+2:n}-Z_{n-k+1:n},Z)|Z_{1:n}>t))\nonumber\\&=\frac{1}{P(Z_{1:n}>t)}\int_{0}^{\infty}P(\min(Z_{n-k+2:n}-Z_{n-k+1:n},Z)>x,Z_{1:n}>t) dx\nonumber\\&
=\frac{n}{P(Z_{1:n}>t)}\int_{0}^{\infty}P(Z_{n-k+2:n}>Z_{n}+x,Z_{1:n}>t,Z_{n-k+1:n}=Z_{n},Z>x) dx\nonumber\\&
=\frac{n}{P(Z_{1:n}>t)}\binom{n-1}{n-k} \times \nonumber\\&
\int_{0}^{\infty}P(Z_{1}>x+Z_{n},\ldots, Z_{k-1}>x+Z_{n},Z>x,t<Z_{k}<Z_{n},\ldots,t<Z_{n-1}< Z_{n},Z_{n}>t) dx\nonumber\\&
=\frac{1}{P(Z_{1:n}>t)B(n-k+1,k)} \times \nonumber\\& \Bigg[\int_{0}^{\infty}\int_{t}^{\infty}\underbrace{\int_{t}^{z_{n}}\ldots\int_{t}^{z_{n}}}_{\text{n-k}}\underbrace{ \int_{z_{n}+x}^{\infty}\ldots\int_{z_{n}+x}^{\infty}}_{\text{k-1}}\int_{x}^{\infty}  \frac{\partial^{n+1} C(F(z_{1}),\ldots,F(z_{n}),G(z))}{\partial F(z_{1})\ldots\partial F(z_{n})\partial G(z)}\times  \nonumber\\&~~~~~~~~~~~~~~~~~~~~~~~~~~~~~~~~~~~~~~~~~~~~~~~~~~~~~~~ \Bigg(\prod_{i=1}^{n} f(z_{i})\Bigg)~ g(z) ~ dzdz_{1} \ldots dz_{k-1}dz_{k}\ldots dz_{n-1}dz_{n}dx\Bigg].
\end{align}
\end{proof}
\begin{remark}
\normalfont
If $Z_{1},\ldots,Z_{n},Z$ are independent, then
\begin{align}
E(T&-t|Z_{1:n}>t)=E(Z_{n-k+1:n}-t|Z_{1:n}>t)+
\frac{1}{\overline{F}^{n}(t)B(n-k+1,k)} \times \nonumber\\& \Bigg[\int_{0}^{\infty}\int_{t}^{\infty}\underbrace{\int_{t}^{z_{n}}\ldots\int_{t}^{z_{n}}}_{\text{n-k}}\underbrace{ \int_{z_{n}+x}^{\infty}\ldots\int_{z_{n}+x}^{\infty}}_{\text{k-1}}\int_{x}^{\infty}\Bigg(\prod_{i=1}^{n} f(z_{i})\Bigg)~ g(z) ~ dzdz_{1} \ldots dz_{k-1}dz_{k}\ldots dz_{n-1}dz_{n}dx\Bigg]\nonumber\\&
=\sum_{m=0}^{n-k}\sum_{i=0}^{m} {n \choose m} {m \choose i}(-1)^{i}\frac{\int_{t}^{\infty} \overline{F}^{n-m+i}(z) dz}{\overline{F}^{n-m+i}(t)}+\frac{1}{B(n-k+1,k)\overline{F}^{n}(t)} \times \nonumber\\&
~~~~~~~~~~~~~~~~~~~~~~~~~~~~~~~~~~~~~~~~~~~~~~~~~~~~~~\int_{0}^{\infty}\int_{t}^{\infty} \overline{F}^{k-1}(z_{n}+x) (F(z_{n})-F(t))^{n-k}\overline{G}(x)f(z_{n})dz_{n}dx
\end{align}
This case is considered in \citet{eryilmaz2012mean}.
\end{remark}
\qed \\
In the next section, we illustrate our theoretical results.
\section{Illustrations}
\normalfont
We consider a 2-out-of-3 system equipped with one cold standby component.
Then, the lifetime of the system is denoted by
\begin{align}\label{2out3}
&T=Z_{2:3}+ \min\{Z_{3:3}-Z_{2:3},Z\}.
\end{align}
 By using equations (\ref{m1}) and (\ref{k8111}),
 \begingroup
\allowdisplaybreaks
\begin{align}\label{a26}
P(T>s)&=1-3C(F(s),F(s),1,1)+2C(F(s),F(s),F(s),1)+\nonumber\\&
6\int_{0}^{s}\int_{0}^{z_{3}}\int_{s}^{\infty}\int_{s-z_{3}}^{\infty}\frac{\partial^{4} C(F(z_{1}),F(z_{2}),F(z_{3}),G(z))}{\partial F(z_{1})\partial F(z_{2})\partial F(z_{3})\partial G(z)}\Bigg(\prod_{i=1}^{3} f(z_{i})\Bigg) g(z) ~ dzdz_{1} dz_{2}dz_{3},
\end{align}
\endgroup
and
\begin{align}
P(Z_{2:3}>s)=1-3C(F(s),F(s),1,1)+2C(F(s),F(s),F(s),1).
\end{align}
Here, we assume that the copula function $C$ is a FGM copula function.
The expression of the four-dimensional FGM  copula  function $C$ is given by
\begingroup
\allowdisplaybreaks
\begin{align}\label{a2}
 C(F(z_{1}),F(z_{2}),F(z_{3}),G(z))=&G(z)\Bigg(\prod_{i=1}^{3}F(z_{i})\Bigg)\Bigg[1+\theta _{11} \sum_{1\leq i<j \leq 3} (1-F(z_{i}))(1-F(z_{j}))+\theta _{12} \sum_{i=1}^{3} (1-\nonumber\\&F(z_{i}))(1-G(z))+\theta _{21} \prod_{i=1}^{3} (1-F(z_{i}))+
 \theta _{22} \sum_{1\leq i<j \leq 3} (1-F(z_{i}))(1-\nonumber\\&F(z_{j}))(1-G(z))+\theta _{31} \Bigg( \prod_{i=1}^{3} (1-F(z_{i}))\Bigg)(1-G(z))\Bigg],
 \end{align}
\endgroup
where the parameters $\theta _{11}, \theta _{12},\theta _{21}, \theta _{22}, \theta _{31}~~(-1 \leq \theta _{11}, \theta _{12},\theta _{21}, \theta _{22}, \theta _{31} \leq 1)$ represents the dependence degrees. \\
Therefore, from equation (\ref{a2}),
\begingroup
\allowdisplaybreaks
\begin{align}\label{a27}
 \frac{\partial^{4} C(F(z_{1}),F(z_{2}),F(z_{3}),G(z))}{\partial F(z_{1})\partial F(z_{2})\partial F(z_{3})\partial G(z)}&
 =1+\theta _{11} \sum_{1\leq i<j \leq 3} (1-2F(z_{i}))(1-2F(z_{j}))+\theta _{12}\sum_{i=1}^{3} (1-2F(z_{i}))(1-\nonumber\\&2G(z))
  +\theta _{21} \prod_{i=1}^{3} (1-2F(z_{i}))+ \theta _{22} \sum_{1\leq i<j \leq 3}(1-2F(z_{i}))(1-2F(z_{j}))(1-\nonumber\\&2G(z))
  +\theta _{31} \Bigg( \prod_{i=1}^{3}(1-2F(z_{i}))\Bigg)(1-2G(z)).
\end{align}
\endgroup
By using equations (\ref{a26})-(\ref{a27}), we have
\begingroup
\allowdisplaybreaks
\begin{align}
P( Z_{2:3}>s)&=1-3F^{2}(s)+2F^{3}(s)-3\theta _{11}F^{2}(s)\overline{F}^{2}(s)(1-2F(s))
+2\theta _{21}F^{3}(s)\overline{F}^{3}(s),
\end{align}
\endgroup
and,
\begingroup
\allowdisplaybreaks
\begin{align}
P(T>s)=P( Z_{2:3}>s)&+6\overline{F}(s)\int_{0}^{s}\Bigg[1-(\theta _{11}-\theta _{22}G(s-z))[F(s)\overline{F}(z)+(F(s)-\overline{F}(z))(1-2F(z))]\nonumber\\&+\theta _{12}(F(s)-2+3F(z))G(s-z)-(\theta _{21}-\theta _{31}G(s-z))(1-2F(z))F(s)\overline{F}(z)\Bigg]\nonumber\\&
\times F(z)\overline{G}(s-z)f(z) dz.
\end{align}
\endgroup
Now, we investigate this results for the following cases:\\
Case I: $F(z)=G(z)=1-e^{-2z},~z>0,$ that is, all the components are have exponential lifetime distributions with parameter 2.\\
Case II: $F(z)=G(z)=1-(1+z)^{-2},~z>0,$ that is, all the components are have Pareto type II lifetime distributions.\\
Case III: $F(z)=G(z)=1-e^{-z^{2}},~z>0,$ that is, all the components are have Weibull lifetime distributions      \\
Figure 1 depict the graphs of $P(Z_{2:3}>s)$, $P(T>s)$ for all the cases.
\begin{figure} [h!]
  \centering
\subfloat[Case I]{\includegraphics[width=0.5\textwidth]{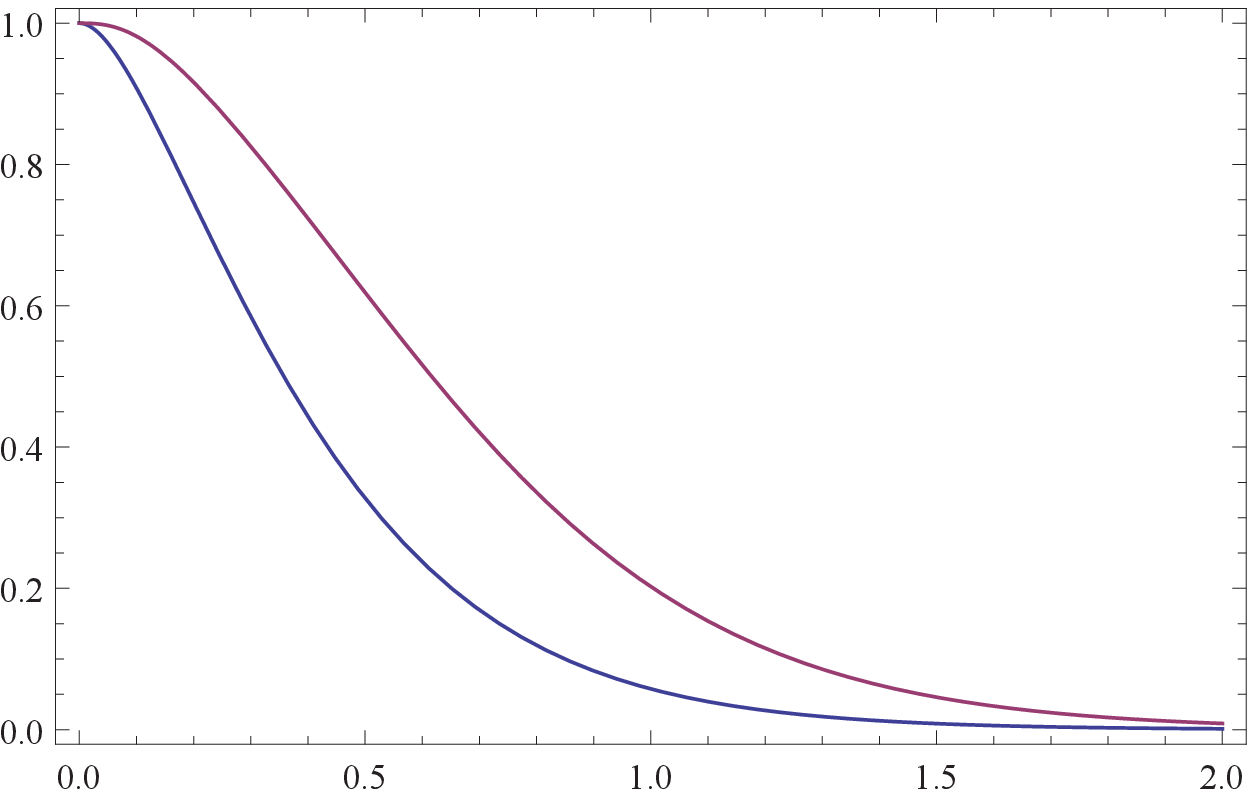}}
 \hfill
 \subfloat[Case II]{\includegraphics[width=0.5\textwidth]{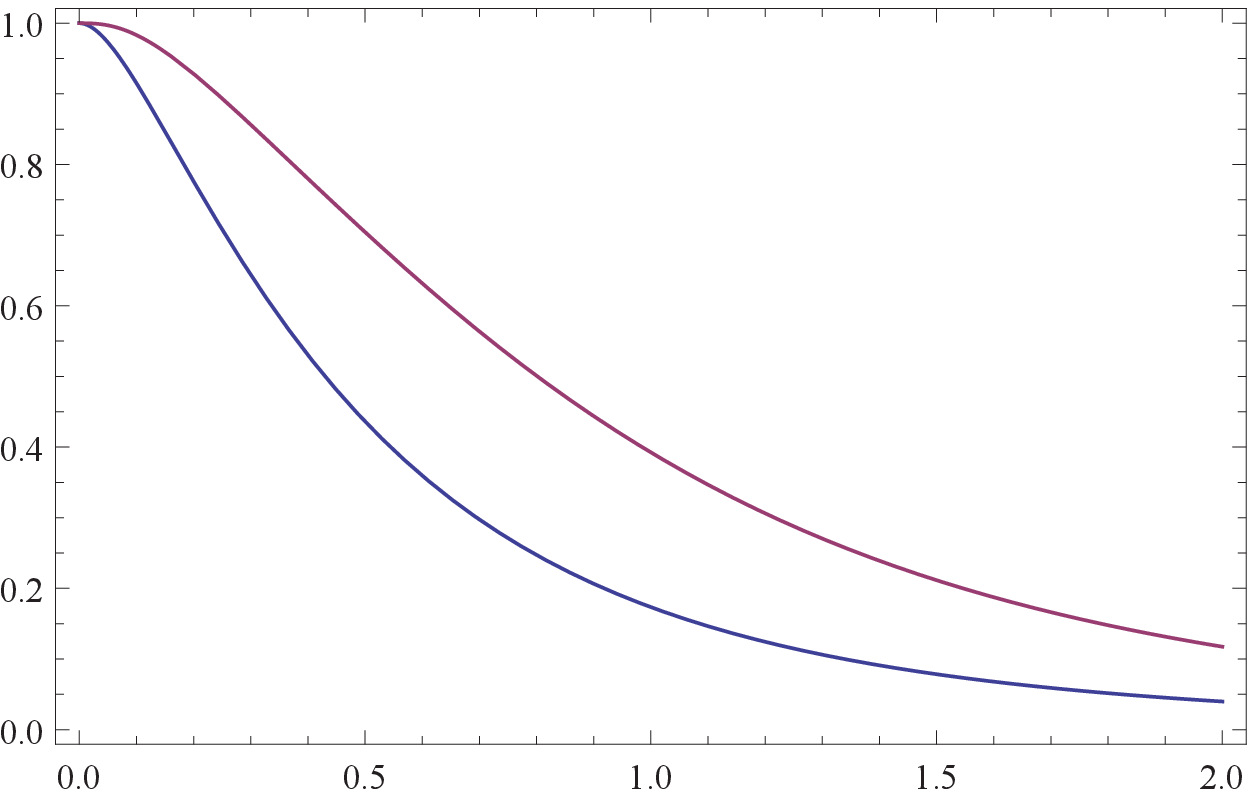}}
 \hfill
 \subfloat[Case III]{\includegraphics[width=0.5\textwidth]{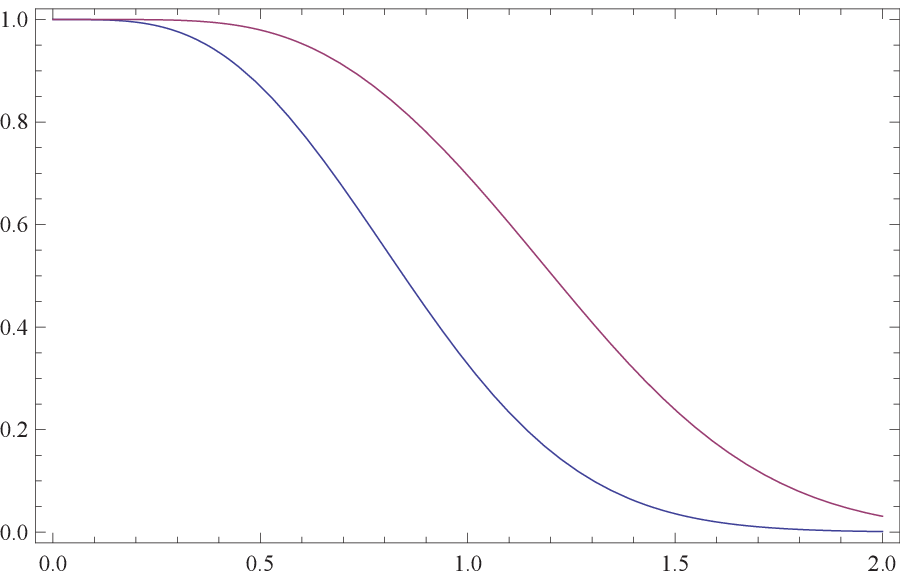}}
 \caption{Plots of  $P(Z_{2:3}>s)$(blue line) and $P(T>s)$(red line)  when $\theta _{11} = .2;~\theta _{12} = .3; ~\theta _{21} = .5;~  \theta _{22}= .6;~ \theta _{31} = .7$.}
\end{figure}
\\
It is intuitive that, mean time to failure as well as manufacturing cost of a system will increase if we use a cold standby into the system.
To evaluate this changes, we first define the cost per unit of time, which are
$$ C_{Z_{n-k+1:n}}=\frac{nc}{E(Z_{n-k+1:n})},$$
 and,
 $$ C_{T}=\frac{(n+1)c}{E(T)}.$$  where $c$ be the acquisition cost for per component. In Tables I-III, we obtain the values of
 $E(Z_{n-k+1:n}),$ $E(T),$ $C_{Z_{n-k+1:n}},$ and $C_{T}$ when $c=1.$
 We see that the cost per unit of time of the system changes significantly (as interpreted in Tables [1-3]).
 \begin{table}[!ht]
\caption{Mean time to failure and mean cost rates for a 2-out-of-3 system with and without a standby component (Case I)}\label{T1}
\centering
\begin{tabular}{lllllllllll}
\hline
\hline
$\theta _{11}$&$\theta _{12}$&$\theta _{21}$&$\theta _{22}$ &$\theta _{31}$&$E(Z_{2:3})$&$E(T)$&$C_{Z_{2:3}}$&$C_{T}$\\[2.5pt]
\hline
 $0$ & $0$ & $0$ & $0$ & $0$ & $0.416667$ & $0.666667$& $7.199994$ & $5.999997$\\[2.5pt]
  $.1$ & $.2$ & $.3$ & $.4$ & $.5$ & $0.424167$ & $0.68125$ & $7.072685$ & $5.871559$\\[2.5pt]
 $.2$ & $.3$ & $.5$ & $.6$ & $.7$ & $0.430000$ & $0.688333$ & $9.302325$ & $5.811140$\\[2.5pt]
\hline
\hline
\end{tabular}
\end{table}
\begin{table}[!ht]
\caption{Mean time to failure and mean cost rates for a 2-out-of-3 system with and without a standby component (Case II)}\label{T1}
\centering
\begin{tabular}{lllllllllll}
\hline
\hline
$\theta _{11}$&$\theta _{12}$&$\theta _{21}$&$\theta _{22}$ &$\theta _{31}$&$E(Z_{2:3})$&$E(T)$&$C_{Z_{2:3}}$&$C_{T}$\\[2.5pt]
\hline
 $0$ & $0$ & $0$ & $0$ & $0$ & $0.6$ & $1.00576$& $5$ & $3.977092$\\[2.5pt]
 $.1$ & $.2$ & $.3$ & $.4$ & $.5$ & $0.615931$ & $1.04269$&  $4.870675$&  $3.836231$ \\[2.5pt]
 $.2$ & $.3$ & $.5$ & $.6$ & $.7$ & $0.629091$ & $1.063$& $4.768785$ & $3.762935$\\[2.5pt]
\hline
\hline
\end{tabular}
\end{table}
\begin{table}[!ht]
\caption{Mean time to failure and mean cost rates for a 2-out-of-3 system with and without a standby component (Case III)}\label{T1}
\centering
\begin{tabular}{lllllllllll}
\hline
\hline
$\theta _{11}$&$\theta _{12}$&$\theta _{21}$&$\theta _{22}$ &$\theta _{31}$&$E(Z_{2:3})$&$E(T)$&$C_{Z_{2:3}}$&$C_{T}$\\[2.5pt]
\hline
 $0$ & $0$ & $0$ & $0$ & $0$ & $0.856644$ & $1.21998$& $3.502054$ & $3.278742$\\[2.5pt]
 $.1$ & $.2$ & $.3$ & $.4$ & $.5$ & $0.863228$ & $1.22371$ & $3.475327$ & $3.268748$\\[2.5pt]
 $.2$ & $.3$ & $.5$ & $.6$ & $.7$ & $0.867908$ & $1.22308$ & $3.456587$ & $3.270432$\\[2.5pt]
\hline
\hline
\end{tabular}
\end{table}
\newpage
Now, we numerically examine how the component's dependence effects the shape of mean residual life functions.
We obtain the plots of $\Psi_{1}(t)$, $\Psi_{2}(t),$ and $\Psi_{3}(t)$ for the system defined by equation (\ref{2out3}) when $F(z)=G(z)=1-e^{-z},~z>0$ for the following cases:\\
Case I:   $\theta _{11} = 0;~\theta _{12} = 0; ~\theta _{21} = 0;~  \theta _{22}= 0;~ \theta _{31} = 0$ \\
Case II: $\theta _{11} = .2;~\theta _{12} = .3; ~\theta _{21} = .5;~  \theta _{22}= .6;~ \theta _{31} = .7$\\
In Figure 2, we see that $\Psi_{3}(t)$ do not change along time for Case I, that is, when components are independent which is due to the memoryless property of exponential distribution. But it is not true when components are dependent (Case II). We see that $\Psi_{3}(t)$ first increase and then decrease for Case II.
  We observe that $\Psi_{1}(t) \leq \Psi_{2}(t) \leq \Psi_{3}(t)$  for all the cases.                              It is also see that $\Psi_{1}(t),$ and $\Psi_{2}(t)$ are decreasing over time.
\begin{figure} [h!]
  \centering
\subfloat[Case I]{\includegraphics[width=0.5\textwidth]{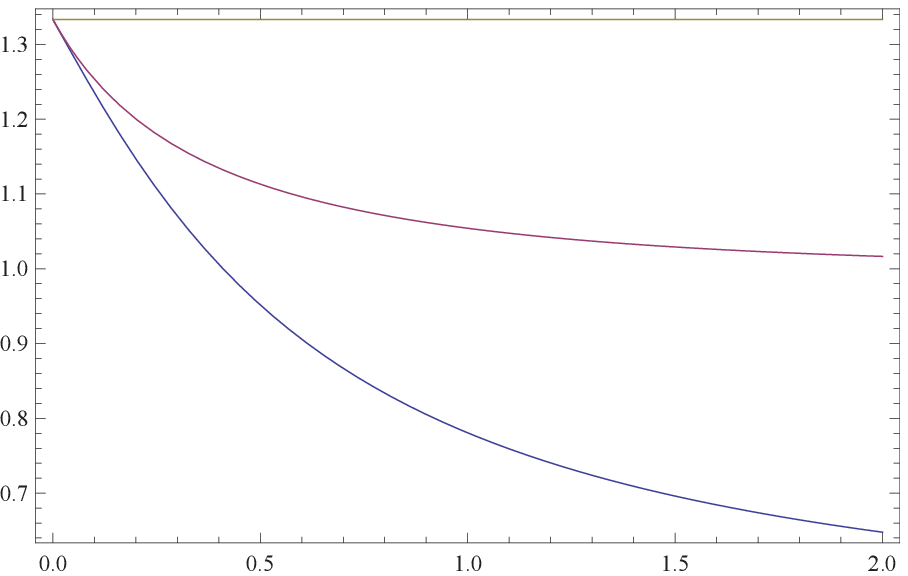}}
 \hfill
 \subfloat[Case II]{\includegraphics[width=0.5\textwidth]{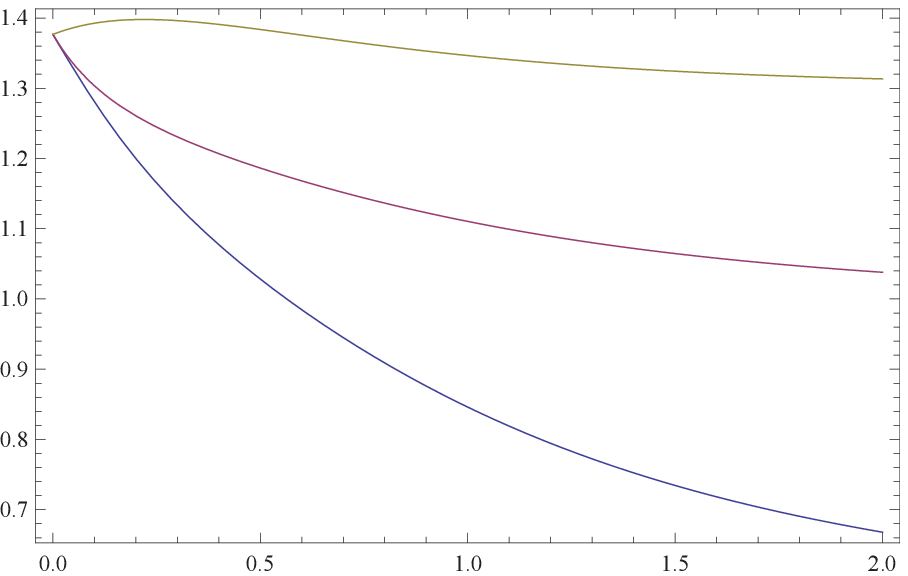}}
 \caption{Plots of  $\Psi_{1}(t)$(blue line), $\Psi_{2}(t)$(red line) and $\Psi_{3}(t)$(brown line) for 2-out of -3 system when $F(z)=G(z)=1-e^{-z},~z>0,$.}
\end{figure}
\section{Conclusions}
Here, we study  a $k$-out-of-$n$ system consisting of $n$ exchangeable components equipped with a cold standby component when all the components are assumed to be dependent.  An exact expression has been obtained for the survival  function of the system in terms of copula functions.  We have also presented explicit expressions for three different mean residual functions of the system. We illustrate the theoretical results  given in this article.
 The results established in this article generalize the results in \citet{eryilmaz2012mean} as
 pointed out in Remark 1, Remark 4 and Corollary 2.
\bibliographystyle{apalike}
\bibliography{ref}
\end{document}